\documentclass{article}
\usepackage{graphicx}
\usepackage{amscd}
\usepackage{graphics,amsmath,amssymb}
\usepackage{amsthm}
\usepackage{amsfonts}
\usepackage{latexsym}
\usepackage{times}

\usepackage[a4paper,left=20mm,right=20mm,top=25mm,bottom=25mm]{geometry}

\newtheorem{theorem}{Theorem}
\theoremstyle{plain}

\newtheorem{corollary}{Corollary}

\newtheorem{definition}{Definition}

\newtheorem{lemma}{Lemma}

\numberwithin{equation}{section}

\usepackage{fancyhdr}
\begin{document}

\begin{center}
\vskip 2cm
{\Large \bf
On the Norms of Circulant and \bigskip $r-$Circulant Matrices With the Hyperharmonic Fibonacci Numbers}\\
\vskip 1cm

\vskip 0.5cm{\large Naim TUGLU$^{a}$ and Can KIZILATE\c{S}$^{b}$} \\
\vskip 0.5cm

$^{a}$Gazi University, Faculty of Art and Science, Department of Mathematics,\smallskip \\
Teknikokullar 06500, Ankara-Turkey \smallskip \\
{\tt naimtuglu@gazi.edu.tr}\\
\vskip 0.5cm
$^{b}$B\"{u}lent Ecevit University, Faculty of Art and Science, Department of Mathematics, \smallskip\\67100, Zonguldak-Turkey \smallskip \\
{\tt cankizilates@gmail.com}
\vskip 0.8cm
\end{center}


\begin{abstract}
In this paper, we study norms of circulant and $r-$circulant matrices involving harmonic Fibonacci and hyperharmonic Fibonacci numbers. We obtain inequalities by using matrix norms.
\end{abstract}

\thispagestyle{empty}
\section{Introduction}

The circulant and $r-$circulant matrices have connection to signal
processing, probability, numerical analysis, coding theory and many other
areas. An $n\times n$ matrix $C_{r}$ is called an $r-$circulant matrix
defined as follows:
	\begin{equation*}
		C_{r}=
			\begin{pmatrix}
				c_{0} & c_{1} & c_{2} & \ldots & c_{n-2} & c_{n-1} \\
				rc_{n-1} & c_{0} & c_{1} & \ldots & c_{n-3} & c_{n-2} \\ 
				rc_{n-2} & rc_{n-1} & c_{0} & \ldots & c_{n-4} & c_{n-3} \\
				\vdots & \vdots & \vdots &  & \vdots & \vdots \\ 
				rc_{1} & rc_{2} & rc_{3} & \ldots & rc_{n-1} & c_{0}
			\end{pmatrix}.
	\end{equation*}%
Since the matrix $C_{r}$ is determined by its row elements and $r$, we denote $C_{r}=Circ(c_{0},c_{1},c_{2},\ldots ,c_{n-1})$. In particular for $r=1$
	\begin{equation*}
		C=
			\begin{pmatrix}
				c_{0} & c_{1} & c_{2} & \ldots & c_{n-2} & c_{n-1} \\ 
				c_{n-1} & c_{0} & c_{1} & \ldots & c_{n-3} & c_{n-2} \\ 
				c_{n-2} & c_{n-1} & c_{0} & \ldots & c_{n-4} & c_{n-3} \\ 
				\vdots & \vdots & \vdots &  & \vdots & \vdots \\ 
				c_{1} & c_{2} & c_{3} & \ldots & c_{n-1} & c_{0}%
			\end{pmatrix}%
	\end{equation*}
is called a circulant matrix and we denote it shortly by $C=Circ(c_{0},c_{1},c_{2},\ldots ,c_{n-1})$. The eigenvalues of $C$ are
	\begin{equation}
		\lambda _{j}=\sum\limits_{i=0}^{n-1}c_{i}(w^{j})^{i}  \label{0}
	\end{equation}%
where $w=e^{\frac{2\pi i}{n}}$ and $i=\sqrt{-1}$.

Many authors investigate on the norms of circulant and $r-$circulant matrices. In \cite{1}, Solak studied the lower and upper bounds for the spectral norms of circulant matrices with classical Fibonacci and Lucas numbers entries. In \cite{2}, Kocer and et al. obtained norms of circulant and semicirculant matrices with Horadams numbers. In \cite{7}, Zhou and et all. gave spectral norms of circulant-type matrices involving binomial coefficients and harmonic numbers.
In \cite{4}, Zhou calculated spectral norms for circulant matrices with binomial coefficients combined with Fibonacci and Lucas numbers entries. In \cite{3}, Shen and Cen have given upper and lower bounds for the spectral norms of $r-$ circulant matrices with classical Fibonacci and Lucas numbers entries. In \cite{5}, Bah\c{s}i and Solak computed the spectral norms of circulant and $r-$circulant matrices with the hyper-Fibonacci and hyper-Lucas numbers. In \cite{6}, Jiang and Zhou studied spectral norms of even-order $r-$ circulant matrices.

Motivated by the above papers, we compute the spectral norms and Euclidean norm of circulant and $r-$circulant matrices with the harmonic and hyperharmonic Fibonacci entries. The scheme of this paper is as follows. In section 2, we present some definitions, preliminaries and lemmas related to our study. In section 3, we calculate spectral norms of circulant matrix with harmonic Fibonacci entries. Moreover we obtain Euclidean norms of $r-$circulant matrices and give lower and upper bounds for the spectral norms of $r-$circulant matrices with harmonic and hyperharmonic Fibonacci entries.


\section{Preliminaries}

The Fibonacci numbers $F_{n}$ are defined by the following recurrence relation for $n\geq 1$,
	\begin{equation*}
		F_{n+1}=F_{n}+F_{n-1}
	\end{equation*}%
where $F_{0}=0$, $F_{1}=1$. In \cite{8}, authors investigated finite sum of the reciprocals of Fibonacci numbers
	\begin{equation*}
		\mathbb{F}_{n}=\sum_{k=1}^{n}\frac{1}{F_{k}},
	\end{equation*}%
which is called harmonic Fibonacci numbers. Then they gave a combinatoric identity related to harmonic Fibonacci numbers as follows:
	\begin{equation}
		\sum_{k=0}^{n-1}F_{k-1}\mathbb{F}_{k}=F_{n}\mathbb{F}_{n}-n.  \label{3}
	\end{equation}
Moreover in \cite{8}, they defined hyperharmonic numbers for $n,r\geq 1$ 
	\begin{equation*}
		\mathbb{F}_{n}^{(r)}=\sum\limits_{k=1}^{n}\mathbb{F}_{k}^{(r-1)}
	\end{equation*}%
where $\mathbb{F}_{n}^{(0)}=\frac{1}{F_{n}}$ and $\mathbb{F}_{0}=0.$ At this point, we give some definitions and lemmas related to our study.

\begin{definition}\label{Definition2.1}
Let $A=(a_{ij})$ be any $m\times n$ matrix. The Euclidean norm of matrix $A$ is
	\begin{equation*}
		\left\Vert A\right\Vert _{E}=\sqrt{\left(\sum_{i=1}^{m}\sum_{j=1}^{n}\left\vert a_{ij}\right\vert ^{2}\right) .}
	\end{equation*}
\end{definition}

\begin{definition}
	Let $A=(a_{ij})$ be any $m\times n$ matrix. The spectral norm of matrix $A$ is
		\begin{equation*}
			\left\Vert A\right\Vert _{2}=\sqrt{\max_{1\leq i\leq n}\lambda _{i}(A^{H}A)},
		\end{equation*}%
	where $\lambda _{i}(A^{H}A)$ is eigenvalue of $A^{H}A$ and $A^{H}$ is conjugate transpose of matrix $A$.
\end{definition}

Then the following inequalities hold for between Euclidean norm and spectral norm as follow;
	\begin{equation}
		\frac{1}{\sqrt{n}}\left\Vert A\right\Vert _{E}\leq \left\Vert A\right\Vert_{2}\leq \left\Vert A\right\Vert _{E},  \label{7}
	\end{equation}%
	\begin{equation}
		\left\Vert A\right\Vert _{2}\leq \left\Vert A\right\Vert _{E}\leq \sqrt{n}\left\Vert A\right\Vert _{2}.  \label{8}
	\end{equation}
\begin{lemma}\label{Lemma2.1} \cite{9}
	 Let $A$ and $B$ be two $m\times n$ matrices. Then we have%
		\begin{equation*}
			\left\Vert A\circ B\right\Vert _{2}\leq \left\Vert A\right\Vert_{2}\left\Vert B\right\Vert _{2}
		\end{equation*}
	where $A\circ B$ is the Hadamard product of $A$ and $B$.
\end{lemma}

\begin{lemma}
	\cite{9} Let $A$ and $B$ be two $n\times m$ matrices. We have%
		\begin{equation*}
			\left\Vert A\circ B\right\Vert _{2}\leq r_{1}(A)c_{1}(B),
		\end{equation*}
	where
		\begin{equation*}
			r_{1}(A)=\max_{1\leq i\leq m}\sqrt{\sum_{j=1}^{n}\left\vert a_{ij}\right\vert ^{2},}
		\end{equation*}%
		\begin{equation*}
			c_{1}(B)=\max_{1\leq j\leq n}\sqrt{\sum_{i=1}^{m}\left\vert b_{ij}\right\vert ^{2}.}
		\end{equation*}
\end{lemma}

\begin{definition}
	\cite{10} Difference operator of $f(x)$ is defined as%
		\begin{equation*}
			\Delta f(x)=f(x+1)-f(x).
		\end{equation*}
\end{definition}

\begin{definition}
	\cite{10} A function $f(x)$ with the property that $\Delta f(x)=g(x)$ is called anti-difference operator of $g(x).$
\end{definition}

\begin{lemma}
	\cite{10} If $\Delta f(x)=g(x)$, then%
		\begin{equation*}
			\sum\limits_{a}^{b}g(x)\delta _{x}=\sum\limits_{x=a}^{b-1}g(x)=f(b)-f(a).
	\end{equation*}
\end{lemma}

\begin{lemma}\cite{10} \label{Lemma2.4}
	 We have
		\begin{equation} \sum_{a}^{b}u(x)\Delta v(x)\delta _{x}=\left. u(x)v(x)\right\vert_{a}^{b+1}-\sum\limits_{a}^{b}v(x+1)\Delta u(x)\delta _{x}.  \label{9}
	\end{equation}
\end{lemma}

\begin{lemma}\cite{10}
	 For $m\neq -1$ we have
		\begin{equation*}
			\sum x^{\underline{m}}\delta _{x}=\frac{x^{\underline{m+1}}}{m+1}
		\end{equation*}%
	where $x^{\underline{m}}=x(x-1)(x-2)\ldots (x-m+1).$
\end{lemma}


\section{Main Results}

\begin{theorem}
\cite{8} Let $C_{1}=Circ(\mathbb{F}_{0},\mathbb{F}_{1},\mathbb{F}_{2},\ldots, \mathbb{F}_{n-1})$ be $n\times n$ circulant matrix. The spectral norm of $C_{1}$ is
	\begin{equation*}
		\left\Vert C_{1}\right\Vert _{2}=n\mathbb{F}_{n}-\sum_{k=0}^{n-1}\frac{k+1}{F_{k+1}}\text{.}
	\end{equation*}
\end{theorem}

\begin{theorem}
	\cite{8} Let $C^{(k)}=Circ(\mathbb{F}_{0}^{(k)},\mathbb{F}_{1}^{(k)},\mathbb{F}_{2}^{(k)},\ldots ,\mathbb{F}_{n-1}^{(k)})$ be $n\times n$ circulant matrix. The spectral norm of $C^{(k)}$ is
		\begin{equation*}
			\left\Vert C^{(k)}\right\Vert _{2}=\mathbb{F}_{n-1}^{(k+1)}\text{.}
		\end{equation*}
\end{theorem}

\begin{theorem}
	The spectral norm of the matrix
		\begin{equation*}
			C=Circ(F_{-1}\mathbb{F}_{0},F_{0}\mathbb{F}_{1},\ldots,F_{n-2}\mathbb{F}_{n-1})
		\end{equation*}
	is
		\begin{equation*}
			\left\Vert C\right\Vert _{2}=F_{n}\mathbb{F}_{n}-n.
		\end{equation*}
\end{theorem}

\begin{proof}
	Since $C$ is a circulant matrix, from the (\ref{0}), for all $t=0,1,\ldots,s-1$
		\begin{equation*}
			\lambda _{t}(C)=\sum\limits_{i=0}^{s-1}F_{i-1}\mathbb{F}_{i}(w^{t})^{i}.
		\end{equation*}
	Then for $t=0$,
		\begin{equation}
			\lambda _{0}(C)=\sum\limits_{i=0}^{s-1}F_{i-1}\mathbb{F}_{i}  \label{a}
		\end{equation}
	and from the (\ref{3}), $\lambda _{0}(C)=F_{n}\mathbb{F}_{n}-n.$ Hence, for $1\leq m\leq n-1$, we have
		\begin{equation}
			\left\vert \lambda _{m}\right\vert =\left\vert\sum\limits_{i=0}^{s-1}F_{i-1}\mathbb{F}_{i}(w^{t})^{i}\right\vert \leq \left\vert \sum\limits_{i=0}^{s-1}F_{i-1}\mathbb{F}_{i}\right\vert \left\vert (w^{t})^{i}\right\vert \leq \sum\limits_{i=0}^{s-1}F_{i-1} \mathbb{F}_{i}.  \label{b}
		\end{equation}
	Since $C$ is a normal matrix, we have%
		\begin{equation}
			\left\Vert C\right\Vert _{2}=\max_{0\leq m\leq n-1}\left\vert \lambda_{m}\right\vert .  \label{c}
		\end{equation}
	From the (\ref{a}),(\ref{b}),(\ref{c}) and (\ref{3}), we have
		\begin{equation*}
			\left\Vert C\right\Vert _{2}=F_{n}\mathbb{F}_{n}-n.
		\end{equation*}
\end{proof}

\begin{corollary}
	We have
	\begin{equation*}
	\sqrt{\sum\limits_{k=0}^{n-1}F_{k-1}^{2}\mathbb{F}_{n}^{2}}\leq F_{n}	\mathbb{F}_{n}-n\leq \sqrt{n\sum\limits_{k=0}^{n-1}F_{k-1}^{2}\mathbb{F}	_{n}^{2}}.
	\end{equation*}
\end{corollary}

\begin{proof}
	The proof is trivial from the Definition \ref{Definition2.1} and the relation between Euclidean norm and spectral norm in (\ref{7}).
\end{proof}

\begin{theorem} \label{Theorem3.5}
		Let $C_{r}^{(k)}=Circ(\mathbb{F}_{0}^{(k)},\mathbb{F}_{1}^{(k)},\ldots , \mathbb{F}_{n-1}^{(k)})$ be $n\times n$ $r-$ circulant matrix. The Euclidean norm of $C_{r}^{(k)}$ is
			\begin{equation*}
				\left\Vert C_{r}^{(k)}\right\Vert _{E}=\left[ \frac{n}{2}\left( 	n+1+(n-1)\left\vert r\right\vert ^{2}\right) \left( \mathbb{F}		_{n}^{(k)}\right) ^{2}-\frac{1}{2}\sum\limits_{s=0}^{n-1}(s+1)\left( 2n+s(\left\vert r\right\vert ^{2}-1)\right) (\mathbb{F}_{s+1}^{(k-1)}+2		\mathbb{F}_{s}^{(k)})\mathbb{F}_{s+1}^{(k-1)}\right] ^{\frac{1}{2}}.
			\end{equation*}
\end{theorem}

\begin{proof}
	From the definition of Euclidean norm we have,
		\begin{eqnarray*}
			\left\Vert C_{r}^{(k)}\right\Vert _{E} &=&\left[ \sum_{s=0}^{n-1}(n-s)\left( \mathbb{F}_{s}^{(k)}\right) ^{2}+\sum_{s=0}^{n-1}s\left\vert r\right\vert^{2}\left( \mathbb{F}_{s}^{(k)}\right) ^{2}\right] ^{\frac{1}{2}} \\
			&=&\left[ \sum_{s=0}^{n-1}(n+s(\left\vert r\right\vert ^{2}-1))\left( \mathbb{F}_{s}^{(k)}\right) ^{2}\right] ^{\frac{1}{2}}.
		\end{eqnarray*}
	Now we will use property of difference operator in Lemma \ref{Lemma2.4} Let $u(s)=\left( \mathbb{F}_{s}^{(k)}\right) ^{2}$ and $\Delta v(s)=n+s(\left\vert r\right\vert ^{2}-1)$. Then using the definition of hyperharmonic Fibonacci numbers we obtain $\Delta u(s)=\mathbb{F} _{s+1}^{(k-1)}(\mathbb{F}_{s+1}^{(k-1)}+2\mathbb{F}_{s}^{(k)})$ and $v(s)=ns+\frac{s^{\underline{2}}}{2}(\left\vert r\right\vert ^{2}-1).$ By using the equation (\ref{9}), we have
		\begin{equation*}
			\left\Vert C_{r}^{(k)}\right\Vert _{E}=\left[ \frac{n}{2}\left(n+1+(n-1)\left\vert r\right\vert ^{2}\right) \left( \mathbb{F}_{n}^{(k)}\right)^{2}-\frac{1}{2}\sum\limits_{s=0}^{n-1}(s+1)\left(2n+s(\left\vert r\right\vert ^{2}-1)\right) (\mathbb{F}_{s+1}^{(k-1)}+2\mathbb{F}_{s}^{(k)})\mathbb{F}_{s+1}^{(k-1)}\right] ^{\frac{1}{2}}.
		\end{equation*}
\end{proof}

\begin{corollary}
	Let $C_{r}=Circ(\mathbb{F}_{0},\mathbb{F}_{1},\ldots ,\mathbb{F}_{n-1})$ be $n\times n$ $r-$circulant matrix. The Euclidean norm of $C_{r}$ is
		\begin{equation*}
			\left\Vert C_{r}\right\Vert _{E}=\left[ \left( n^{2}+\frac{n^{\underline{2}}}{2}(\left\vert r\right\vert ^{2}-1)\right) \mathbb{F}_{n}^{2}-\sum\limits_{s=0}^{n-1}\left(n(s+1)+\frac{(s+1)^{^{\underline{2}}}}{2}(\left\vert r\right\vert ^{2}-1)\right) \left( 2\mathbb{F}_{s}+\frac{1}{F_{s+1}}\right) \frac{1}{F_{s+1}}\right] ^{\frac{1}{2}}.
		\end{equation*}
\end{corollary}

\begin{proof}
	It is clear that the proof can be completed if we take $k=1$ in Theorem \ref{Theorem3.5}
\end{proof}

\begin{corollary}
	\cite{8} Let $C_{1}=Circ(\mathbb{F}_{0},\mathbb{F}_{1},\ldots ,\mathbb{F}_{n-1})$ be $n\times n$ matrix. The Euclidean norm is%
		\begin{equation*}
			\left\Vert C_{1}\right\Vert _{E}=\left[ n^{2}\mathbb{F}_{n}^{2}-n\sum\limits_{k=0}^{n-1}\frac{k+1}{F_{k+1}}\left( 2\mathbb{F}_{k}+\frac{1}{F_{k+1}}\right) \right] ^{\frac{1}{2}}.
		\end{equation*}
\end{corollary}

\begin{proof}
	It is easily seen that the proof can be completed if we take $k=r=1$ in Theorem \ref{Theorem3.5}
\end{proof}

Now we give upper and lower bounds for the spectral norms of $r-$circulant matrices.

\begin{theorem}\label{Theorem3.8}
	Let $C_{r}^{(k)}=Circ(\mathbb{F}_{0}^{(k)},\mathbb{F}_{1}^{(k)},\ldots ,\mathbb{F}_{n-1}^{(k)})$ be $n\times n$ $r-$circulant matrix. \\
	$i)$ If $\left\vert r\right\vert \geq 1$, then
		\begin{equation*}
			\frac{1}{\sqrt{n}}\mathbb{F}_{n-1}^{(k+1)}\leq \left\Vert C_{r}^{(k)}\right\Vert _{2}\leq \left\vert r\right\vert \sqrt{n-1} \: \mathbb{F}_{n-1}^{(k+1)}.
		\end{equation*} \\
	$ii)$ If $\left\vert r\right\vert <1$, then
		\begin{equation*}
			\frac{\left\vert r\right\vert}{\sqrt{n}}\mathbb{F}_{n-1}^{(k+1)}\leq \left\Vert C_{r}^{(k)}\right\Vert _{2}\leq \sqrt{n-1}\mathbb{F}_{n-1}^{(k+1)}.
		\end{equation*}
\end{theorem}
\begin{proof}
	Since the matrix
		\begin{equation*}
			C_{r}^{(k)}=
				\begin{pmatrix}
					\mathbb{F}_{0}^{(k)} & \mathbb{F}_{1}^{(k)} & \mathbb{F}_{2}^{(k)} & \ldots & \mathbb{F}_{n-2}^{(k)} & \mathbb{F}_{n-1}^{(k)} \\ r\mathbb{F}_{n-1}^{(k)} & \mathbb{F}_{0}^{(k)} & \mathbb{F}_{1}^{(k)} & \ldots & \mathbb{F}_{n-3}^{(k)} & \mathbb{F}_{n-2}^{(k)} \\ \vdots & \vdots & \vdots &  & \vdots & \vdots \\ r\mathbb{F}_{2}^{(k)} & r\mathbb{F}_{3}^{(k)} & r\mathbb{F}_{4}^{(k)} & \ldots & \mathbb{F}_{0}^{(k)} & \mathbb{F}_{1}^{(k)} \\ r\mathbb{F}_{1}^{(k)} & r\mathbb{F}_{2}^{(k)} & r\mathbb{F}_{3}^{(k)} & \ldots & r\mathbb{F}_{n-1}^{(k)} & \mathbb{F}_{0}^{(k)}
				\end{pmatrix},
		\end{equation*}
	we have
		\begin{equation*}
			\left\Vert C_{r}^{(k)}\right\Vert _{E}=\sqrt{\sum_{s=0}^{n-1}(n-s)\left( \mathbb{F}_{s}^{(k)}\right) ^{2}+\sum_{s=0}^{n-1}s\left\vert r\right\vert^{2}\left( \mathbb{F}_{s}^{(k)}\right) ^{2}.}
		\end{equation*}
	$i)$ In \cite{8}, for the sum of the squares of hyperharmonic Fibonacci numbers, we have
		\begin{equation}
			\frac{1}{\sqrt{n}}\mathbb{F}_{n-1}^{(r+1)}\leq \sqrt{\sum_{k=0}^{n-1}\left( \mathbb{F}_{k}^{(r)}\right) ^{2}}\leq \mathbb{F}_{n-1}^{(r+1)}.  \label{10}
		\end{equation}
	Since $\left\vert r\right\vert \geq 1$ and by (\ref{10}), we have
		\begin{equation*}
			\left\Vert C_{r}^{(k)}\right\Vert _{E}\geq \sqrt{\sum_{s=0}^{n-1}(n-s)\left( \mathbb{F}_{s}^{(k)}\right) ^{2}+\sum_{s=0}^{n-1}s\left( \mathbb{F}_{s}^{(k)}\right) ^{2}}\geq \sqrt{n\sum_{s=0}^{n-1}\left( \mathbb{F}_{s}^{(k)}\right) ^{2}}\geq \mathbb{F}_{n-1}^{(k+1)}.
		\end{equation*}
	From (\ref{7})
		\begin{equation*}
			\frac{1}{\sqrt{n}}\mathbb{F}_{n-1}^{(k+1)}\leq \left\Vert C_{r}^{(k)}\right\Vert _{2}.
		\end{equation*}
	On the other hand, let the matrices $A$ and $B$ be as
		\begin{equation*}
			A=
				\begin{pmatrix}
					\mathbb{F}_{0}^{(k)} & 1 & 1 & \ldots & 1 & 1 \\ r & \mathbb{F}_{0}^{(k)} & 1 & \ldots & 1 & 1 \\ \vdots & \vdots & \vdots &  & \vdots & \vdots \\ r & r & r & \ldots & \mathbb{F}_{0}^{(k)} & 1 \\ r & r & r & \ldots & r & \mathbb{F}_{0}^{(k)}
				\end{pmatrix}
		\end{equation*}
	and
		\begin{equation*}
			B=
				\begin{pmatrix}
					\mathbb{F}_{0}^{(k)} & \mathbb{F}_{1}^{(k)} & \mathbb{F}_{2}^{(k)} & \ldots & \mathbb{F}_{n-2}^{(k)} & \mathbb{F}_{n-1}^{(k)} \\ \mathbb{F}_{n-1}^{(k)} & \mathbb{F}_{0}^{(k)} & \mathbb{F}_{1}^{(k)} & \ldots & \mathbb{F}_{n-3}^{(k)} & \mathbb{F}_{n-2}^{(k)} \\ \vdots & \vdots & \vdots &  & \vdots & \vdots \\ \mathbb{F}_{2}^{(k)} & \mathbb{F}_{3}^{(k)} & \mathbb{F}_{4}^{(k)} & \ldots & \mathbb{F}_{0}^{(k)} & \mathbb{F}_{1}^{(k)} \\ \mathbb{F}_{1}^{(k)} & \mathbb{F}_{2}^{(k)} & \mathbb{F}_{3}^{(k)} & \ldots & \mathbb{F}_{n-1}^{(k)} & \mathbb{F}_{0}^{(k)}
				\end{pmatrix},		
		\end{equation*}
	That is $C_{r}^{(k)}=A\circ B.$ Then we obtain
		\begin{equation*}
			r_{1}(A)=\max_{1\leq i\leq n}\sqrt{\sum\limits_{j=1}^{n}\left\vert a_{ij}\right\vert ^{2}}=\sqrt{\sum\limits_{j=1}^{n}\left\vert a_{nj}\right\vert ^{2}}=\sqrt{( n-1) \left\vert r\right\vert ^{2}}
		\end{equation*}
	and
		\begin{equation*}
			c_{1}(B)=\max_{1\leq j\leq n}\sqrt{\sum\limits_{i=1}^{n}\left\vert b_{ij}\right\vert ^{2}}=\sqrt{\sum\limits_{i=1}^{n}\left\vert b_{in}\right\vert ^{2}}=\sqrt{\sum_{s=0}^{n-1}\left( \mathbb{F}_{s}^{(k)}\right) ^{2}.}
		\end{equation*}
	Hence, from the (\ref{10}) and Lemma \ref{Lemma2.1}, we have
		\begin{equation*}
			\left\Vert C_{r}^{(k)}\right\Vert _{2}\leq \left\vert r\right\vert \sqrt{ n-1} \, \mathbb{F}_{n-1}^{(k+1)}.
		\end{equation*}
	Thus, we have
		\begin{equation*}
			\frac{1}{\sqrt{n}}\mathbb{F}_{n-1}^{(k+1)}\leq \left\Vert C_{r}^{(k)}\right\Vert _{2}\leq \left\vert r\right\vert \sqrt{ n-1} \, \mathbb{F}_{n-1}^{(k+1)}
		\end{equation*}
	$ii)$ Since $\left\vert r\right\vert <1$ and from the (\ref{10}), we have
		\begin{eqnarray*}
			\left\Vert C_{r}^{(k)}\right\Vert _{E} &=&\sqrt{\sum_{s=0}^{n-1}(n-s)\left( \mathbb{F}_{s}^{(k)}\right) ^{2}+\sum_{s=0}^{n-1}s\left\vert r\right\vert ^{2}\left( \mathbb{F}_{s}^{(k)}\right) ^{2}} \\ &\geq &\sqrt{\sum_{s=0}^{n-1}(n-s)\left\vert r\right\vert ^{2}\left( \mathbb{F}_{s}^{(k)}\right) ^{2}+\sum_{s=0}^{n-1}s\left\vert r\right\vert ^{2}\left(\mathbb{F}_{s}^{(k)}\right) ^{2}} \\ &=&\left\vert r\right\vert \sqrt{n\sum_{s=0}^{n-1}\left( \mathbb{F}_{s}^{(k)}\right) ^{2}} \\ &\geq &\left\vert r\right\vert \mathbb{F}_{n-1}^{(k+1)}. 
		\end{eqnarray*} 
	From( \ref{7}),
		\begin{equation*}
			\frac{\left\vert r\right\vert}{\sqrt{n}}\mathbb{F}_{n-1}^{(k+1)}\leq \left\Vert C_{r}^{(k)}\right\Vert _{2}.
		\end{equation*}
	On the other hand, let the matrices $A$ and $B$ be as%
		\begin{equation*}
			A=
				\begin{pmatrix}
					\mathbb{F}_{0}^{(k)} & 1 & 1 & \ldots & 1 & 1 \\ r & \mathbb{F}_{0}^{(k)} & 1 & \ldots & 1 & 1 \\ \vdots & \vdots & \vdots &  & \vdots & \vdots \\ r & r & r & \ldots & \mathbb{F}_{0}^{(k)} & 1 \\ r & r & r & \ldots & r & \mathbb{F}_{0}^{(k)}
				\end{pmatrix}
		\end{equation*}
	and
		\begin{equation*}
			\begin{pmatrix}
				\mathbb{F}_{0}^{(k)} & \mathbb{F}_{1}^{(k)} & \mathbb{F}_{2}^{(k)} & \ldots& \mathbb{F}_{n-2}^{(k)} & \mathbb{F}_{n-1}^{(k)} \\ \mathbb{F}_{n-1}^{(k)} & \mathbb{F}_{0}^{(k)} & \mathbb{F}_{1}^{(k)} & \ldots& \mathbb{F}_{n-3}^{(k)} & \mathbb{F}_{n-2}^{(k)} \\ \vdots & \vdots & \vdots &  & \vdots & \vdots \\ \mathbb{F}_{2}^{(k)} & \mathbb{F}_{3}^{(k)} & \mathbb{F}_{4}^{(k)} & \ldots& \mathbb{F}_{0}^{(k)} & \mathbb{F}_{1}^{(k)} \\ \mathbb{F}_{1}^{(k)} & \mathbb{F}_{2}^{(k)} & \mathbb{F}_{3}^{(k)} & \ldots& \mathbb{F}_{n-1}^{(k)} & \mathbb{F}_{0}^{(k)}
			\end{pmatrix}.
		\end{equation*}
	Thus $C_{r}^{(k)}=A\circ B.$ Then we obtain
		\begin{equation*}
			r_{1}(A)=\max_{1\leq i\leq n}\sqrt{\sum\limits_{j=1}^{n}\left\vert a_{ij}\right\vert ^{2}}=\sqrt{\left( \mathbb{F}_{0}^{(k)}\right) ^{2}+n-1}=\sqrt{n-1}
		\end{equation*}
	and
		\begin{equation*}
			c_{1}(B)=\max_{1\leq j\leq n}\sqrt{\sum\limits_{i=1}^{n}\left\vert b_{ij}\right\vert ^{2}}=\sqrt{\sum_{s=0}^{n-1}\left( \mathbb{F}_{s}^{(k)}\right) ^{2}}
		\end{equation*}
	Therefore, from the (\ref{10}) and Lemma \ref{Lemma2.1}, we have
		\begin{equation*}
			\left\Vert C_{r}^{(k)}\right\Vert _{2}\leq \sqrt{n-1}\,\mathbb{F}_{n-1}^{(k+1)}.
		\end{equation*}
	Thus, we have
		\begin{equation*}
			\frac{\left\vert r\right\vert}{\sqrt{n}}\mathbb{F}_{n-1}^{(k+1)}\leq \left\Vert C_{r}^{(k)}\right\Vert _{2}\leq \sqrt{n-1}\,\mathbb{F}_{n-1}^{(k+1)}.
		\end{equation*}
\end{proof}

\begin{corollary}
	Let $C_{r}=Circ(\mathbb{F}_{0},\mathbb{F}_{1},\ldots ,\mathbb{F}_{n-1})$ be $n\times n$ $r-$circulant matrix. \\
	$i)$ If $\left\vert r\right\vert \geq 1$, then
		\begin{equation*}
			\frac{1}{\sqrt{n}}\mathbb{F}_{n-1}^{(2)}\leq \left\Vert C_{r}\right\Vert_{2}\leq \left\vert r\right\vert \sqrt{ n-1} \, \mathbb{F}_{n-1}^{(2)}.
		\end{equation*}
	$ii)$ If $\left\vert r\right\vert <1$, then%
		\begin{equation*}
			\frac{\left\vert r\right\vert}{\sqrt{n}}\mathbb{F}_{n-1}^{(2)}\leq \left\Vert C_{r}\right\Vert_{2}\leq \sqrt{n-1}\mathbb{F}_{n-1}^{(2)}.
		\end{equation*}
\end{corollary}
\begin{proof}
It is easily seen that the proof can be completed if we take $k=1$ in Theorem \ref{Theorem3.8}
\end{proof}



\begin{thebibliography}{99}
\bibitem{1} S. Solak, On the norms of circulant matrices with the Fibonacci and Lucas numbers, \textit{Appl. Math. Comput}. \textbf{160} (2005) 125--132.

\bibitem{2} E. G. Kocer, T. Mansour, N. Tuglu, Norms of Circulant and Semicirculant Matrices with Horadams Numbers, \textit{Ars Combinatoria}, \textbf{85} (2007) 353-359.  

\bibitem{3} S. Q. Shen, J.M. Cen, On the bounds for the norms of $r-$circulant matrices with Fibonacci and Lucas numbers, \textit{Appl. Math. Comput}. \textbf{216 }(2010) 2891--2897

\bibitem{4} J. Zhou, The Identical Estimates of Spectral Norms for Circulant Matrices with Binomial Coefficients Combined with Fibonacci Numbers and Lucas Numbers Entries, \textit{J. Funct. Spaces Appl. }\textbf{Article ID 518913}, 7 (2014).

\bibitem{5} M. Bahsi, S. Solak, On the norms of $r-$circulant matrices with the hyper-Fibonacci and Lucas numbers, \textit{J. Math. Inequal. }\textbf{8 (4)},(2014), 693-705.

\bibitem{6} Z. Jiang, J. Zhou, A note on spectral norms of even-order r-circulant matrices,\textit{Appl. Math. Comput. }\textbf{250 }(2015) 368--371.

\bibitem{7} J. Zhou, X. Chen, Z. Jiang, The Explicit Identities for Spectral Norms of Circulant-Type Matrices Involving Binomial Coefficients and Harmonic Numbers, \textit{Math. Probl.Eng. }\textbf{Article ID 672398}, 5 (2014).

\bibitem{8} N. Tuglu, C. K\i z\i late\c{s}, S. Kesim, On the harmonic and hyperharmonic Fibonacci numbers (submitted).

\bibitem{9} R. A. Horn, C. R. Johnson, \textit{Topics in Matrix Analysis}, Cambridge University Press, 1991.

\bibitem{10} R. Graham, D. Knuth, K. Patashnik, O. \textit{Concrete Mathematics}, Addison Wesley, 1989.

\bibitem{11} R. A. Horn, C. R. Johnson, \textit{Matrix Analysis}, Cambridge University Press, Cambridge, UK, 1985.

\end{thebibliography}
\end{document}